\documentclass[a4paper,11pt]{amsart}

\usepackage[latin1]{inputenc}
\usepackage[cyr]{aeguill}
\usepackage[english]{babel}

\usepackage{amsmath,amsfonts,amssymb,amsthm}
\usepackage{hyperref,array}
\usepackage{bbold}
\usepackage{enumitem,footnote}
\usepackage{graphicx,color,tikz}
\usepackage{newtxtext}
\usepackage{ulem}

\hypersetup{colorlinks=true,    linkcolor=blue,
citecolor=red,       filecolor=BrickRed,       urlcolor=darkgreen
}

\newcommand{\E}{\mathcal{E}}
\newcommand{\EE}{\mathbb{E}}
\newcommand{\PP}{\mathbb{P}}
\newcommand{\RR}{\mathbb{R}}
\renewcommand{\geq}{\geqslant}
\renewcommand{\leq}{\leqslant}
\renewcommand{\epsilon}{\varepsilon}

\def \sclr#1#2{\langle #1,#2\rangle}

\newtheorem{thm}{Theorem} 
\newtheorem{cor}[thm]{Corollary}
\newtheorem{lem}[thm]{Lemma}
\newtheorem{rem}[thm]{Remark}

\title[Survival probabilities of random walks with drift inside a pyramid]{Random walks with drift inside a pyramid: convergence rate for the survival probability}

\author{Rodolphe Garbit}
\address{Universit\'e d'Angers, CNRS, Laboratoire Angevin de Recherche en Math\'ematiques, SFR MATHSTIC, 49000 Angers, France}
\email{rodolphe.garbit@univ-angers.fr}
\email{raschel@math.cnrs.fr}
\author{Kilian Raschel}

\thanks{This work has received funding from the European Research Council (ERC) under the European Union's Horizon 2020 research and innovation programme under the Grant Agreement No.\ 759702, from Centre Henri Lebesgue, programme ANR-11-LABX-0020-0 and from the project DeRerumNatura, programme ANR-19-CE40-0018}

\keywords{Random walks in cones; Survival probabilities; Laplace transform; Pyramids; Overshoots of random walks}

\subjclass{60G50,60G40,60E10,52B11}

\date{\today}

\begin{document}

\begin{abstract}
We consider multidimensional random walks in pyramids, which by definition are cones formed by finite intersections of half-spaces. The main object of interest is the survival probability $\PP(\tau>n)$, $\tau$ denoting the first exit time from a fixed pyramid. When the drift belongs to the interior of the cone, the survival probability sequence converges to the non-exit probability $\PP(\tau=\infty)$, which is positive. In this note, we quantify the speed of convergence, and prove that the exponential rate of convergence may be computed by means of a certain min-max of the Laplace transform of the random walk increments. We illustrate our results with various examples.
\end{abstract}

\maketitle 

\section{Introduction and main results}

\subsection*{A glimpse of our results}
For a $d$-dimensional random walk $(S_n)_{n\geq 0}$ with integrable and independent increments $X_n=S_n-{S_{n-1}}$ having  common distribution $\mu$, we consider the survival probabilities
\begin{equation}
\label{eq:gen_func}
  \PP^x(\tau>n),
\end{equation}
where $\tau$ denotes the first exit time from a given cone $K$, i.e.
\begin{equation*}
   \tau=\inf\{n>0 \vert S_n\notin K\},
\end{equation*}   
and $\PP^x$ is a probability distribution under which the random walk starts at $S_0=x$, with $x\in K$.

When the drift $m=\EE X_1$ belongs to the interior $K^o$ of the cone $K$, the non-exit probability $\PP^x(\tau=\infty)$, which is the limit of the sequence \eqref{eq:gen_func}, is positive (see \cite[Lem.~8]{GaRa22} for example). In this note, our main result quantifies the speed of convergence in the following way:
\begin{equation}
\label{eq:two-term_estimate}
   \PP^x(\tau>n)=\PP^x(\tau=\infty)+\rho^n B_n,
\end{equation}
where the exponential rate $\rho\in (0,1)$, and $B_n$ satisfies $\sqrt[n]{B_n}\to 1$ and $B_n\to 0$.
The precise statement is given in Theorem~\ref{thm:inner_drift} below.
The rate $\rho$ is computed in terms of a certain min-max of the Laplace transform of $\mu$.

In the special case of small step walks in $\mathbb Z^d$ (i.e.~when the support of $\mu$ is a subset of $\{-1,0,1\}^d$) in the orthant $\mathbb Z_+^d$, this result was previously obtained in~\cite[Thm~4]{GaRa22}. 

In the present paper, we consider general probability distributions $\mu$ with all exponential moments and polyhedral convex cones, i.e.~finite intersection of half-spaces, which for short we will call pyramids; see examples on Figures~\ref{fig:example_dual} and~\ref{fig:examples_introduction}.

One initial motivation to obtain formula~\eqref{eq:two-term_estimate} is the following consequence on the generating function
\begin{equation}
\label{eq:survival_series}
    \phi(z)=\sum_{n=0}^{\infty} \PP^x(\tau>n) z^n.
\end{equation}
If the survival probabilities behave as in~\eqref{eq:two-term_estimate}, then the generating function~\eqref{eq:survival_series} can not be a rational function, as shown in \cite{GaRa22} using singularity analysis. The question of proving rationality (and various refinements, such as algebraicity) of generating functions as above is inspired by the combinatorial work \cite{BoMi10}, where the rational nature of series as in \eqref{eq:survival_series} is used to measure the complexity of the associated combinatorial problem.

\subsection*{Technical assumptions}
In order to present the hypotheses under which we shall prove our main results, we introduce two objects, through which the exponential rate $\rho$ in \eqref{eq:two-term_estimate} will be determined:
\begin{itemize}
\item the Laplace transform $L$ of the increment distribution $\mu$:
\begin{equation*}
\label{eq:def_Laplace_transform}
     L(t)=\EE \bigl( e^{\sclr{t}{X_1}}\bigr) =\int_{\mathbb R^d}e^{\sclr{t}{y}}\mu({d}y),
\end{equation*}
\item the dual cone $K^*$ associated with $K$ (see Figure~\ref{fig:example_dual} for an example):
\begin{equation*}
\label{eq:def_dual_cone}
     K^*=\{x\in\mathbb R^d \vert \sclr{x}{y}\geq 0 \mbox{ for all } y\in K\}.
\end{equation*}
Obviously, $K^*$ is a closed convex cone.
\end{itemize}

\begin{figure}
\centering
\begin{tikzpicture}
\begin{scope}[scale=0.5]
\draw[red!30,fill=red!30] (0.1,0) -- (6.5,0) -- (5.54,3.44);
\draw[red!30,fill=red!30,domain=0:32.5] plot ({6.5*cos(\x)}, {6.5*sin(\x)});
\draw[blue!30,fill=blue!30,domain=33.5:90] plot ({5.5*cos(\x)}, {5.5*sin(\x)});
\draw[blue!30,fill=blue!30,domain=-56.5:0] plot ({5.5*cos(\x)}, {5.5*sin(\x)});
\draw[blue!30,fill=blue!30] (0,-0.08) -- (5.5,-0.08) -- (3.04,-4.59);
\draw[blue!30,fill=blue!30] (0,0.1) -- (0,5.5) -- (4.65,3);
\draw[white, thick] (-1.5,-1.5) grid (6.5,6.5);
\draw[white, thick] (-1.5,1.5) grid (6.5,-4.5);
\draw[blue,densely dotted,domain=0:32.5] plot ({5.5*cos(\x)}, {5.5*sin(\x)});
\draw[blue!30,thick,domain=33:90] plot ({5.5*cos(\x)}, {5.5*sin(\x)});
\draw[blue!30,thick,domain=-56.5:0] plot ({5.5*cos(\x)}, {5.5*sin(\x)});
\draw[red,thick] (0,0) -- (5.5,3.47);
\draw[red,thick] (0,0.03) -- (6.5,0.03);
\draw[blue,thick] (0,0) -- (3.04,-4.59);
\draw[blue,thick] (0.03,0) -- (0.03,5.5);
\draw[->,thick] (0,-5) -- (0,7.5);
\draw[->,thick] (-1.5,0) -- (7.5,0);
\draw[green] (0,0.3) -- (0.3,0.3);
\draw[green] (0.3,0) -- (0.3,0.3);
\draw[green] (0.34,0.22) -- (0.56,-0.11);
\draw[green] (0.22,-0.33) -- (0.56,-0.11);
\node[above right] at (1.75,0.1) {\textcolor{red}{$K$}};
\node[above right] at (1.75,-1.1) {\textcolor{blue}{$K^*$}};
\end{scope}
\end{tikzpicture}
\caption{A cone $K$ (in red) and its dual cone $K^*$ (in blue)}
\label{fig:example_dual}
\end{figure}
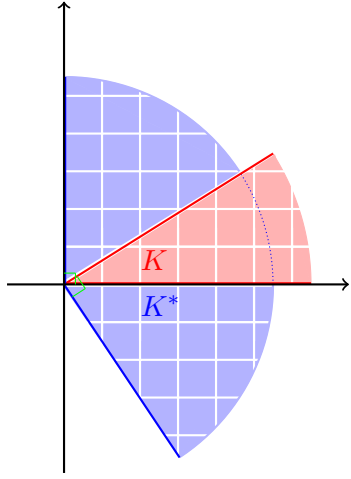

We will also use extensively the notation $D_u$ for the closed half-space with inner normal $u\in \mathbb{S}^{d-1}$, i.e.
\begin{equation*}
    D_u=\{y\in\RR^d \vert \sclr{y}{u}\geq 0\}.
\end{equation*}
Note that $K\subset D_u$ if and only if $u\in K^*$.
Disclaimer: when using the notation $D_u$, it is understood that $u$ belongs to the sphere $\mathbb{S}^{d-1}$ (in particular $u\neq 0$).

Throughout this paper, we make the following assumptions on the cone $K$ and the distribution $\mu$ of the random walk increments:
\begin{enumerate}[label=(A\arabic{*}),ref=(A\arabic{*})]
   \item\label{hyp:A1}[Cone] The cone $K$ is a \textit{finite intersection} of closed half-spaces $D_u$,  where $u$ varies in a finite subset $S$ of $\mathbb{S}^{d-1}$, hence $K=\cap_{u\in S} D_{u}$, and it has a non-empty interior. We call this type of cone a closed \textit{pyramid}. See Figure~\ref{fig:examples_introduction} for examples.
   \item\label{hyp:A2}[Adaptation to the dimension] The random walk is truly $d$-dimensional, i.e.~there is no $u\not=0$ such that $\sclr{u}{X_1}=0$ almost surely.
   \item\label{hyp:A3}[Adaptation to the cone] The random walk started at zero can reach the interior $K^o$ of the cone: there exists $k>0$ such that $\PP^0(\tau> k, S_k \in K^o)>0$.
   \item\label{hyp:A4}[Exponential moments] The random walk increments have all exponential moments. In other words, the Laplace transform is finite everywhere on $\RR^d$.
   \item\label{hyp:A5}[Non-triviality] The random walk is not trapped in the cone: $\PP(X_1\in K)<1$. (If $\PP(X_1 \in K)=1$, then $\PP^x(\tau>n)=1$ for all $n$ and there is nothing more to say.)
\end{enumerate}
We call $m=\EE X_1=\int y \mu({d}y)$ the drift.

\begin{figure}
    \centering
    \includegraphics[scale=0.33]{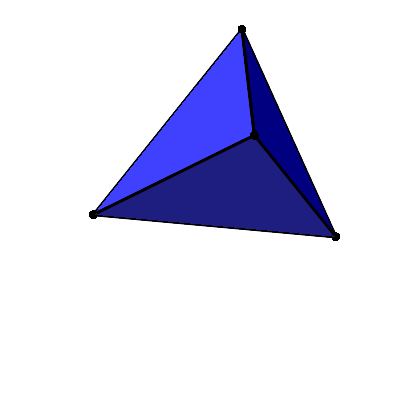}
    \includegraphics[scale=0.33]{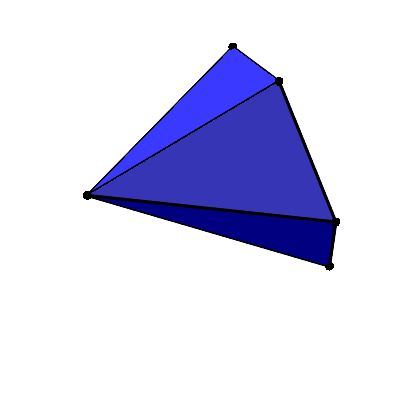}
    \includegraphics[scale=0.33]{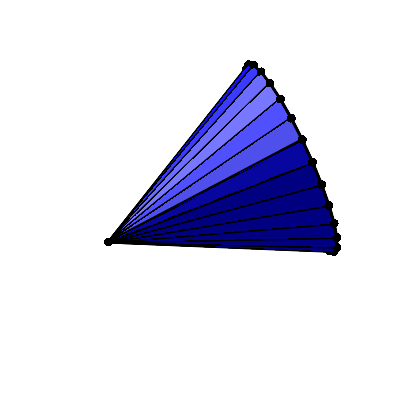}
    \vspace{-12mm}
    \caption{Examples of (truncated) pyramids in dimension $3$}
    \label{fig:examples_introduction}
\end{figure}

\subsection*{Precise statements}
Our main result is the following:

\begin{thm}
\label{thm:inner_drift}
Assume hypotheses \ref{hyp:A1}--\ref{hyp:A5} above, $K=\cap_{u\in S} D_{u}$, and $m\in K^o$. The subset $S'\subset S$ of directions $u$ such that the equation $L(su)=1$ has a solution $s=s_u<0$ is non-empty, and
\begin{equation*}
   \PP^x(\tau>n)=\PP^x(\tau=\infty)+\rho^n B_n,
\end{equation*}
with
\begin{equation}
\label{eq:value_rho}
   \rho=\max_{u\in S'}\min_{z\in K^*}L(t_u+z)\in (0,1),
\end{equation}
where $t_u=s_u u$  and $B_n$ satisfies $\sqrt[n]{B_n}\to 1$ together with $B_n\to 0$.
\end{thm}

\begin{rem}
\label{rem:autre_carac_de_S'}
Under the hypotheses of Theorem \ref{thm:inner_drift}, the set $S'$ is also characterized as being the subset of all $u$ in $S$ such that $\PP(X_1\in D_u)<1$. See Lemma \ref{lem:new_drift_on_gamma} for a proof.
\end{rem}

\begin{rem}
\label{rem:choice}
It might not be clear at first sight why the expression \eqref{eq:value_rho} for $\rho$ doesn't depend on the set $S$ as long as $K=\cap_{u\in S} D_{u}$. The reason is that one set $S_0$ among such $S$ is minimal with respect to inclusion, and any vector in $S$ can be written as a non-negative linear combination of vectors in $S_0$. This combined with further basic properties of the convex function $t\mapsto \min_{z\in K^*}L(t+z)$ makes it possible to deduce that the maximum on $\{t_u \vert u\in S'\}$ is reached on  $\{t_u \vert u\in S_0'\}$. The set $S_0$ can be characterized as the set of extremal directions of $K^*$. This is explained in Appendix~\ref{sec:supp_mat};  its reading is not necessary for the understanding of the proof of Theorem~\ref{thm:inner_drift}.
\end{rem}

Using similar singularity analysis technniques as in \cite{GaRa22}, the estimate obtained in Theorem~\ref{thm:inner_drift} yields the following: 
\begin{cor}
Assume hypotheses \ref{hyp:A1}--\ref{hyp:A5} above, $K=\cap_{u\in S} D_{u}$, and $m\in K^o$. Then the generating function~\eqref{eq:survival_series} is not rational.
\end{cor}

Although we are not able to extend Theorem~\ref{thm:inner_drift} to the case of cones which are not pyramids, we conjecture that the same conclusion holds, provided the formula~\eqref{eq:value_rho} for $\rho$ be replaced by
\begin{equation*}
   \rho=\max_{t\in M}\min_{z\in K^*}L(t+z),
\end{equation*}
where $M=(-K^*)\cap\{L=1\}\setminus\{0\}$.

The case of a drift $m\notin K^o$ is considered in \cite{GaRa22}, and our proof of Theorem~\ref{thm:inner_drift} is based on that previous result, which we now state for convenience:

\begin{thm}[Theorem~3 in \cite{GaRa22}]
\label{thm:exp_rate_and_zero_term}
Assume hypotheses \ref{hyp:A1}--\ref{hyp:A4} above and that $L$ is coercive on $K^*$. If $m\notin K^o$, then 
\begin{equation*}
   \PP^x(\tau>n)=\rho^n B_n,
\end{equation*}
where
\begin{equation*}
    \rho=\min_{z\in K^*} L(z)\in (0,1],
\end{equation*}
and $B_n$ satisfies $\sqrt[n]{B_n}\to 1$ and $B_n\to 0$.
Moreover $\rho<1$ if and only if $m\notin \partial K$.
\end{thm}
In Theorem~\ref{thm:exp_rate_and_zero_term} and throughout the manuscript, a non-negative, convex function $f:\mathbb R^d\to \mathbb R$ is said to be coercive on a cone $C$ if $f(x)\to +\infty$ as $\Vert x\Vert\to +\infty$, $x\in C$.


\section{Examples of application of Theorem~\ref{thm:inner_drift}}

In this section, we give various illustrations of Theorem~\ref{thm:inner_drift}.

\subsection{Small step examples with uniform distribution}
All examples presented in Table~\ref{tab:liste} below are small step walks in the plane, confined to the cone $K=\mathbb R_+^2$, with uniform distribution. By definition, two-dimensional small step models have a support included in the set of the eight nearest neighbors $\{-1,0,1\}^2\setminus \{(0,0)\}$.

In Table~\ref{tab:liste}, we use the notation of Theorem~\ref{thm:inner_drift}; for example, $s_{(1,0)}$ denotes the unique negative point such that $L(s_{(1,0)},0)=1$. We further introduce
\begin{equation}
    \label{eq:def_rho_10}
    \rho_{(1,0)} = \min_{z\in \mathbb R_+^2}L(t_{(1,0)}+z),
\end{equation}
with $t_{(1,0)} = (s_{(1,0)},0)$. The quantities $s_{(0,1)}$ and $\rho_{(0,1)}$ are defined similarly. The rate $\rho$ is as in \eqref{eq:value_rho} and satisfies $\rho=\max \{\rho_{(1,0)},\rho_{(0,1)}\}$.

In the list of $79$ intrinsically different models of walks in the quarter plane established in \cite{BoMi10}, exactly $12$ of them have a drift inside of the quadrant. For each of these $12$ models, we compute $s_{(1,0)}$, $\rho_{(1,0)}$, $s_{(0,1)}$, $\rho_{(0,1)}$, and finally the rate $\rho$ appearing in Theorem~\ref{thm:inner_drift}. For the first model, the rate $\rho$ is already computed in \cite[Prop.~9]{MiRe-09}. The second, third and fourth rates are obtained in \cite[Thm~3.1]{MeMi-14}.

Among these $12$ models, four of them have a support included in a half-plane. These models are considered in the four first rows of Table~\ref{tab:liste}. Notice that the first column of Table~\ref{tab:liste} represents the steps of the random walk; it is implicitly assumed that the transition probabilities are uniform. For example, the first model has jump probabilities $\frac{1}{3}$ in the directions $(-1,1)$, $(1,1)$ and $(1,-1)$.

The remaining $8$ models have simultaneously a drift in the cone and a support which is not included in any half-plane. These models are represented on Table~\ref{tab:liste} as well.




\begin{figure}[ht!]
    \centering
    \includegraphics[scale=0.4]{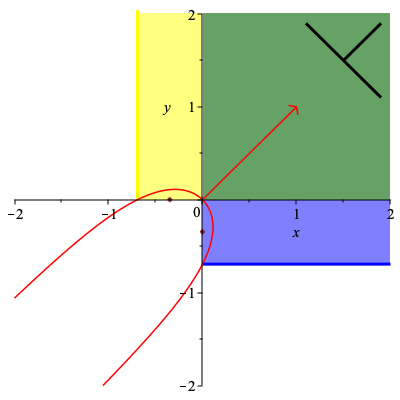}
    \includegraphics[scale=0.4]{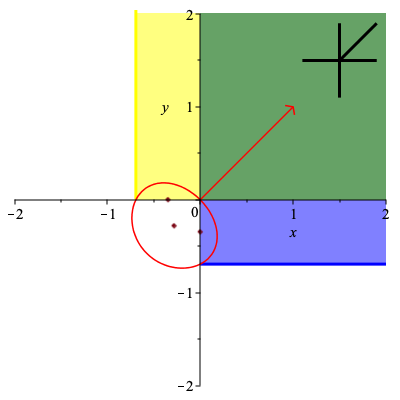}
    \caption{On the left (right), the first example of Table~\ref{tab:liste} (the nineth example of Table~\ref{tab:liste}). We represent the level set $\{t\in \mathbb R^2 \vert L(t) = 1\}$ (red color), the drift (red), the translated cones $t_{(1,0)}+\mathbb R_+^2$ (yellow) and $t_{(0,1)}+\mathbb R_+^2$ (blue); the intersection of the two previous translated quadrants is the positive quarter plane (green). On the left display, the two points stand for the global minima in $t_{(1,0)}+\mathbb R_+^2$ and $t_{(0,1)}+\mathbb R_+^2$. On the right display, the same minima points are drawn, as well as the global minimum on $\mathbb R^2$.}
    \label{fig:examples_domain_1}
\end{figure}
\begin{figure}
    \centering
    \includegraphics[scale=0.4]{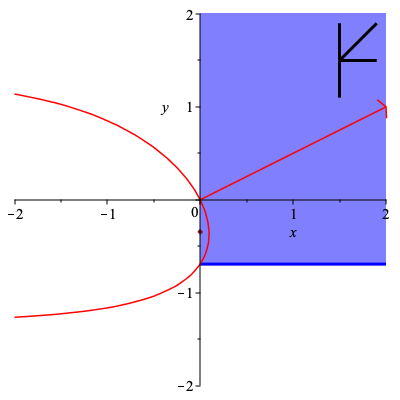}
    \includegraphics[scale=0.4]{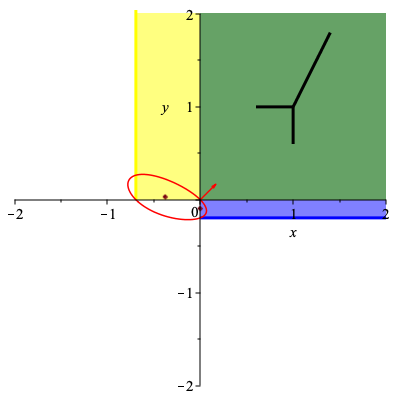}
    \caption{This figure contains two examples, which are complementary to those of Figure~\ref{fig:examples_domain_1}. On the left, a model with step set included in a half-space. The Laplace transform does not admit a minimum in the direction $(1,0)$, so according to Theorem~\ref{thm:inner_drift}, $S'=\{(0,1)\}$ and $\rho$ is computed as the global minimum on $t_{(0,1)}+\mathbb R_+^2$. On the right, the model  with Laplace transform $L(x,y) = \frac{e^{-x}}{6} + \frac{e^{-y}}{2} + \frac{e^{x+2y}}{3}$.}
    \label{fig:examples_domain_2}
\end{figure}

\footnotesize

\begin{table}
\begin{center}
\begin{tabular}{ | m{1.5cm} | m{1.5cm} | m{1.5cm} | m{1.5cm} | m{1.5cm} | m{1.5cm} | } 
 \hline
 Model & $s_{(1,0)}$ & $\rho_{(1,0)}$ & $s_{(0,1)}$ & $\rho_{(0,1)} $ & $\rho$ \\
 \hline 
\begin{tikzpicture}[scale=.35] 
    \draw[->,white] (1,2) -- (0,-2);
    \draw[->,white] (1,-2) -- (0,2);
    \draw[->] (0,0) -- (-1,1);
    \draw[->] (0,0) -- (1,-1);
    \draw[->] (0,0) -- (1,1);
   \end{tikzpicture} & $-\log 2$ & $ \frac{2\sqrt{2}}{3}$ & $-\log 2$ & $\frac{\sqrt{8}}{3}$  & $\frac{\sqrt{8}}{3}$ \cite{MiRe-09}\\ 
  \hline
 \begin{tikzpicture}[scale=.35] 
    \draw[->,white] (1,2) -- (0,-2);
    \draw[->,white] (1,-2) -- (0,2);
    \draw[->] (0,0) -- (-1,1);
    \draw[->] (0,0) -- (1,-1);
    \draw[->] (0,0) -- (1,0);
    \draw[->] (0,0) -- (0,1);
   \end{tikzpicture} & $-\log 2$  & $\frac{1+2\sqrt{2}}{4}$ & $-\log 2$ & $\frac{1+2\sqrt{2}}{4}$ &  $\frac{1+2\sqrt{2}}{4}$ \cite{MeMi-14} \\ 
  \hline
 \begin{tikzpicture}[scale=.35] 
  \draw[->,white] (1,2) -- (0,-2);
    \draw[->,white] (1,-2) -- (0,2);
    \draw[->] (0,0) -- (-1,1);
    \draw[->] (0,0) -- (1,-1);
    \draw[->] (0,0) -- (1,0);
    \draw[->] (0,0) -- (0,1);
    \draw[->] (0,0) -- (1,1);
   \end{tikzpicture} & $-\log3$ &$\frac{1+2\sqrt{3}}{5}$ & $-\log3$ &$\frac{1+2\sqrt{3}}{5}$ & $\frac{1+2\sqrt{3}}{5}$ \cite{MeMi-14}\\ 
 \hline
 \begin{tikzpicture}[scale=.35] 
  \draw[->,white] (1,2) -- (0,-2);
    \draw[->,white] (1,-2) -- (0,2);
    \draw[->] (0,0) -- (-1,1);
    \draw[->] (0,0) -- (1,-1);
    \draw[->] (0,0) -- (0,1);
    \draw[->] (0,0) -- (1,1);
   \end{tikzpicture} & $-\log2$  &$\frac{1+2\sqrt{2}}{4}$ & $-\log3$ &$\frac{\sqrt{3}}{2}$ &  $\frac{1+2\sqrt{2}}{4}$ \cite{MeMi-14}\\ 
 \hline
  \hline 
\begin{tikzpicture}[scale=.35] 
    \draw[->,white] (1,2) -- (0,-2);
    \draw[->,white] (1,-2) -- (0,2);
    \draw[->] (0,0) -- (1,0);
    \draw[->] (0,0) -- (0,1);
    \draw[->] (0,0) -- (-1,-1);
    \draw[->] (0,0) -- (1,1);
   \end{tikzpicture} & $-\log 2$ & $ \frac{\sqrt{8}}{3}$ & $-\log 2$ & $\frac{\sqrt{8}}{3}$  & $\frac{\sqrt{8}}{3}$ \\ 
  \hline
\begin{tikzpicture}[scale=.35] 
    \draw[->,white] (1,2) -- (0,-2);
    \draw[->,white] (1,-2) -- (0,2);
    \draw[->] (0,0) -- (1,1);
    \draw[->] (0,0) -- (1,-1);
    \draw[->] (0,0) -- (-1,0);
    \draw[->] (0,0) -- (0,1);
   \end{tikzpicture} &  $-\log2$ & $\frac{1+2\sqrt{2}}{4}$ & $-\log2$ &$\frac{1+2\sqrt{2}}{4}$ & $\frac{1+2\sqrt{2}}{4}$\\ 
  \hline
   \begin{tikzpicture}[scale=.35] 
    \draw[->,white] (1,2) -- (0,-2);
    \draw[->,white] (1,-2) -- (0,2);
    \draw[->] (0,0) -- (1,1);
    \draw[->] (0,0) -- (-1,1);
    \draw[->] (0,0) -- (0,1);
    \draw[->] (0,0) -- (1,0);
    \draw[->] (0,0) -- (0,-1);
   \end{tikzpicture} & $-\log2$ & $\frac{2+2\sqrt{2}}{5}$ & $-\log 3$ & $\frac{1+2\sqrt{3}}{5}$ & $\frac{2+2\sqrt{2}}{5}$ \\ 
 \hline
 \begin{tikzpicture}[scale=.35] 
    \draw[->,white] (1,2) -- (0,-2);
    \draw[->,white] (1,-2) -- (0,2);
    \draw[->] (0,0) -- (1,1);
    \draw[->] (0,0) -- (-1,1);
    \draw[->] (0,0) -- (0,1);
    \draw[->] (0,0) -- (1,-1);
    \draw[->] (0,0) -- (0,-1);
   \end{tikzpicture} & $-\log 2$ & $\frac{2+2\sqrt{2}}{5}$ & $\log(\frac{2}{3})$& $\frac{2\sqrt{6}}{5}$ & $\frac{2\sqrt{6}}{5}$ \\ 
 \hline
\begin{tikzpicture}[scale=.35] 
    \draw[->,white] (1,2) -- (0,-2);
    \draw[->,white] (1,-2) -- (0,2);
    \draw[->] (0,0) -- (1,1);
    \draw[->] (0,0) -- (-1,0);
    \draw[->] (0,0) -- (0,1);
    \draw[->] (0,0) -- (1,0);
    \draw[->] (0,0) -- (0,-1);
   \end{tikzpicture} & $-\log2$ & $\frac{2+2\sqrt{2}}{5}$ &$-\log2$ & $\frac{2+2\sqrt{2}}{5}$& $\frac{2+2\sqrt{2}}{5}$\\ 
 \hline
   \begin{tikzpicture}[scale=.35] 
    \draw[->,white] (1,2) -- (0,-2);
    \draw[->,white] (1,-2) -- (0,2);
    \draw[->] (0,0) -- (1,1);
    \draw[->] (0,0) -- (-1,1);
    \draw[->] (0,0) -- (1,-1);
    \draw[->] (0,0) -- (1,0);
    \draw[->] (0,0) -- (0,1);
    \draw[->] (0,0) -- (-1,0);
   \end{tikzpicture} & $\log(\frac{2}{3})$ & $\frac{1+2\sqrt{6}}{6}$ & $-\log3$ & $\frac{1+\sqrt{3}}{3}$ & $\frac{1+2\sqrt{6}}{6}$\\ 
 \hline
 \begin{tikzpicture}[scale=.35] 
    \draw[->,white] (1,2) -- (0,-2);
    \draw[->,white] (1,-2) -- (0,2);
    \draw[->] (0,0) -- (1,1);
    \draw[->] (0,0) -- (-1,1);
    \draw[->] (0,0) -- (1,-1);
    \draw[->] (0,0) -- (1,0);
    \draw[->] (0,0) -- (0,1);
    \draw[->] (0,0) -- (-1,-1);
   \end{tikzpicture} & $\log(\frac{2}{3})$  & $\frac{1+2\sqrt{6}}{6}$ & $\log(\frac{2}{3})$ & $\frac{1+2\sqrt{6}}{6}$& $\frac{1+2\sqrt{6}}{6}$\\ 
 \hline
   \begin{tikzpicture}[scale=.35] 
    \draw[->,white] (1,2) -- (0,-2);
    \draw[->,white] (1,-2) -- (0,2);
    \draw[->] (0,0) -- (1,1);
    \draw[->] (0,0) -- (-1,1);
    \draw[->] (0,0) -- (1,-1);
    \draw[->] (0,0) -- (1,0);
    \draw[->] (0,0) -- (0,1);
    \draw[->] (0,0) -- (-1,0);
    \draw[->] (0,0) -- (0,-1);
   \end{tikzpicture} & $\log(\frac{2}{3})$ & $\frac{2+2\sqrt{6}}{7}$ &$\log(\frac{2}{3})$ & $\frac{2+2\sqrt{6}}{7}$ & $\frac{2+2\sqrt{6}}{7}$ \\ 
 \hline
\end{tabular}
\end{center}
\caption{Some important quadrant walk models considered in \cite{BoMi10,MiRe-09,MeMi-14}}
\label{tab:liste}
\end{table}

\normalsize

\subsection{A weighted small step example} 

We now look at a two-dimensional example with Laplace transform
\begin{equation}
\label{eq:Laplace_transform_singular}
    L(x,y) = p_{-1,1}e^{-x+y}+p_{0,1}e^{y}+p_{1,1}e^{x+y}+p_{1,0}e^{x}+p_{1,-1}e^{x-y},
\end{equation}
which is just a weighted version of the third step set on Table~\ref{tab:liste}. We assume that the drift is in the interior of the quadrant, i.e.
\begin{equation}
\label{eq:conditionsurledrift}
p_{-1,1} < p_{1,1}+p_{1,0}+p_{1,-1}\quad \mbox{ and } \quad p_{1,-1} < p_{-1,1}+ p_{0,1} +p_{1,1}
\end{equation}
and that $p_{0,1}+p_{1,1}+p_{1,0}> 0$, so that the walk be truly $2$-dimensional (hypothesis \ref{hyp:A2}). If $p_{-1,1} = p_{1,-1} = 0$, then the walk is trapped in the cone and $\tau=\infty$ almost surely. If both of $p_{-1,1},p_{1,-1}$ are non-zero, then applying Theorem~\ref{thm:inner_drift}, we shall prove that 
\begin{equation}
    \label{eq:value_rho_weighted_singular}
    \rho = \max\bigl\{ p_{0,1}+2\sqrt{p_{-1,1}(p_{1,1}+p_{1,0}+p_{1,-1})},p_{1,0}+2\sqrt{p_{1,-1}(p_{-1,1}+p_{0,1}+p_{1,1})}\bigr\}.
\end{equation}
The formula \eqref{eq:value_rho_weighted_singular} is a generalization of \cite[Thm~3.1]{MeMi-14} to weighted, non-symmetric step sets with Laplace transform \eqref{eq:Laplace_transform_singular}.

If one of $p_{-1,1},p_{1,-1}$ is zero (say $p_{-1,1}$) and the other one is non-zero, then
\begin{equation}
    \label{eq:value_rho_weighted_singular_bis}
    \rho = p_{1,0}+2\sqrt{p_{1,-1}(p_{0,1}+p_{1,1})}.
\end{equation}
\begin{proof}[Proof of \eqref{eq:value_rho_weighted_singular}]
Applying Theorem~\ref{thm:inner_drift} yields
    $\rho=\max\bigl\{ \rho_{(1,0)},\rho_{(0,1)}\bigr\}$, with $\rho_{(1,0)}$ given by \eqref{eq:def_rho_10}, and $\rho_{(0,1)}$ computed symmetrically. In order to derive \eqref{eq:value_rho_weighted_singular}, it is sufficient to show that $\rho_{(1,0)}=p_{0,1}+2\sqrt{p_{-1,1}(p_{1,1}+p_{1,0}+p_{1,-1})}$. 
We first observe that $\rho_{(1,0)}$ is necessarily reached at some boundary point of $t_{(1,0)}+\mathbb{R}^2_+$. Indeed, if $\rho_{(1,0)}$ were reached at an interior point of $t_{(1,0)}+\mathbb{R}^2_+$, then it would be a global minimum of $L$ on $\mathbb R^2$, which does not exist due to the fact that the step set is included in a half-space. Hence it is enough to compute
\begin{equation*}
    \alpha=\min\{ L(s_{(1,0)},y) \vert y \geq 0\}\quad\mbox{ and }\quad \beta=\min\{ L(x,0) \vert x\in [s_{(1,0)}, \infty)\}.
\end{equation*}
The equation $L(x,0)=1$ writes
\begin{equation}
\label{eq:pourtrouverszero}
p_{-1,1}e^{-x}+p_{0,1}+\left(p_{1,1}+p_{1,0}+p_{1,-1}\right) e^{x}=1,
\end{equation}
and is solved by
\begin{equation*}
    x=0\quad \mbox{ and }\quad x=s_{(1,0)}=\ln\left(\frac{p_{-1,1}}{p_{1,1}+p_{1,0}+p_{1,-1}}\right)<0.
\end{equation*}
Now we compute the partial derivative:
\begin{equation*}
\frac{\partial L}{\partial y}(s_{(1,0)},0)  =p_{-1,1}e^{-s_{(1,0)}}+p_{0,1}+p_{1,1}e^{s_{(1,0)}}-p_{1,-1}e^{s_{(1,0)}}
= 1-(p_{0,1}+2p_{1,-1})e^{s_{(1,0)}},
\end{equation*}
where the last equality is obtained thanks to \eqref{eq:pourtrouverszero}.
Because of \eqref{eq:conditionsurledrift}, we have
\begin{equation*}
    p_{0,1}+2p_{1,-1}<p_{0,1}+p_{1,-1} + p_{-1,1}+ p_{0,1} +p_{1,1}=1,
\end{equation*}
and $e^{s_{(1,0)}}<1$. Therefore $\frac{\partial L}{\partial y}(s_{(1,0)},0)>0$ and by convexity the partial function 
$y \mapsto L(s_{(1,0)},y)$ is non-decreasing for $y\geq 0$. By consequence $\alpha=1$.
Accordingly, 
\begin{equation*}
    \rho_{(1,0)} = \beta = \min\{ L(x,0) \vert x\in \mathbb{R}\}.
\end{equation*}
A straightforward computation then shows that the global minimum of a function of the form $ae^{-x}+b+ce^{x}$ is reached at $e^x = \sqrt{\frac{a}{c}}$ and takes the value $b+2\sqrt{ac}$. The proof of~\eqref{eq:value_rho_weighted_singular} is completed.

The formula~\eqref{eq:value_rho_weighted_singular_bis} would be proved similarly, using that $S'=\{(0,1)\}$.
\end{proof}

\subsection{Irrelevance of the location of the drift}

In this paragraph, we would like to illustrate the following fact: the position of the drift (in particular its distance to the boundary)\ does not in general determine which point in $S'$ will give the rate $\rho$.

Let us take three examples in the case of the quarter plane. For the first model in Table~\ref{tab:liste}, the model is symmetric (about the first diagonal), with drift $(\frac{1}{3},\frac{1}{3})$, and one has $\rho = \rho_{(1,0)} = \rho_{(0,1)}$. Consider now the fourth model in Table~\ref{tab:liste}. Its drift is $(\frac{1}{4},\frac{1}{2})$, closer to the vertical axis. The rate is given by $\rho_{(1,0)}$, as shown in Table~\ref{tab:liste}.
Finally, look at the model represented on Figure~\ref{fig:examples_domain_2}, which has a drift of the form $(\frac{1}{6},\frac{1}{6})$. Easy computations show that $\rho_{(1,0)}  \approx 0.97$, which turns out to be the global minimum of $L$ on $\mathbb R^2$, while $\rho_{(0,1)}  \approx 0.99$. By continuity w.r.t.\ the parameters, this last example could be modified to get an example with a drift slightly directed to the vertical axis, but for which the rate $\rho$ would be actually equal to $\rho_{(0,1)}$.

\subsection{Normal distribution}
\label{subsec:normal_dist}
Here we consider the case where the step distribution $\mu$ is a standard normal distribution on $\mathbb{R}^d$ with mean $m$. The Laplace transform is then given by
\begin{equation}
\label{eq:laplace_gaussian}
   L(t)=\exp\left(\sclr{t}{m}+\frac{\Vert t\Vert^2}{2}\right)=\exp\left(\frac{\Vert t+m\Vert^2-\Vert m\Vert^2}{2}\right).
\end{equation}
Let us first recall an explicit expression for the minimum of $L$ on the closed convex cone $K^*$. It is clearly reached when $t$ is the projection of $-m$ on $K^*$, and then
\begin{equation*}
   \Vert t+m\Vert=d(-m,K^*)=d(m, K^{\sharp}),
\end{equation*}
where $K^{\sharp}=-K^*$ is the \textit{polar cone} of $K$. By Moreau's decomposition theorem
\begin{equation*}
   d(m,K^{\sharp})^2+d(m, K)^2=\Vert m \Vert^2,
\end{equation*}
therefore the minimum of $L$ on $K^*$ is $\exp\bigl(-\frac{1}{2}d(m,K)^2\bigr)$.

Now consider the set $\Gamma=\{t\in\mathbb{R}^2 \vert L(t)=1\}$. From \eqref{eq:laplace_gaussian}, we see that $\Gamma$ is the circle $\mathcal{C}(-m,\Vert m\Vert)$ with center at $-m$ and radius $\Vert m\Vert$, see Figure~\ref{fig:examples_domain_3}.
For $t\in\Gamma$, we have
\begin{equation*}
L(t+z)=\exp\left(\frac{\Vert z+ t+m\Vert^2-\Vert t+ m\Vert^2}{2}\right)
\end{equation*}
since $\Vert t+m\Vert^2=\Vert m\Vert^2$.
The function $z\mapsto L(t+z)$ is thus the Laplace transform of a standard normal distribution with mean $t+m$, and it follows that
\begin{equation*}
   \min_{z\in K^*}L(t+z)=\exp\left(-\frac{1}{2}d(t+m,K)^2\right).
\end{equation*}
For a closed pyramid $K=\cap_{u\in S}D_u$ containing $m$ in its interior, the set $S'$ is equal to $S$ (see Remark \ref{rem:autre_carac_de_S'}) and the rate $\rho$ is given by
\begin{equation*}
\rho = \max_{u\in S}\exp\left(-\frac{1}{2}d(t_u+m,K)^2\right)
    = \exp\left(-\frac{1}{2}\min_{u\in S}d(z_u, K)^2\right),
\end{equation*}
where $z_u=t_u+m$ is characterized as being the unique intersection point of the circle $\mathcal{C}(0,\Vert m\Vert)$ and the half-line $\{su+m \vert s<0\}$.

For example, consider the cone $K=\{\rho (\cos \theta, \sin \theta) \vert \rho \geq 0, \theta \in [0, \alpha]\}$ in the plane and $m=r(\cos \beta, \sin\beta)$, with $0<\beta<\alpha\leq \pi$.
Then $K=D_u\cap D_v$ with $u=(0,1)$ and $v=(\sin\alpha,-\cos\alpha)$.
Some computations show that $z_u=r(\cos\beta, -\sin\beta)$ and 
$z_v=r(\cos(2\alpha-\beta), \sin(2\alpha-\beta))$. From this, we obtain $\rho=e^{-d^2/2}$ with 
\begin{equation*}
   d= r \min\big\{\sin \beta, \sin(\alpha-\beta)\bigr\}.
\end{equation*}
Note that in this example, $d$ equals the distance between $m$ and $\partial K$.

\begin{figure}
    \centering
    \includegraphics[scale=0.4]{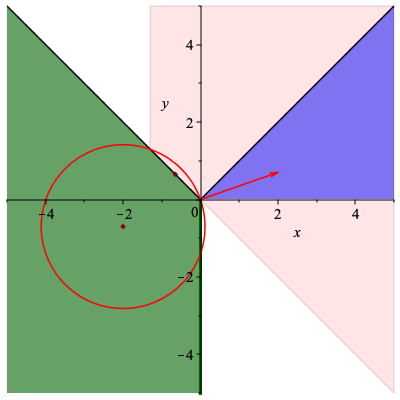}
    \caption{This figure summarizes the results presented in Subsection~\ref{subsec:normal_dist}. The equation $L(t)=1$ describes a circle with center at $c=-m$ and radius $r=\Vert m\Vert$. The cone $K$ is represented in blue and the polar cone in green. In pink, we have represented one of the two domains $t_u+K^*$.}
    \label{fig:examples_domain_3}
\end{figure}

\section{Proof of Theorem \ref{thm:inner_drift}}

\subsection{Sketch of the proof}

Let $K=\cap_{u\in S} D_{u}$ be a closed pyramid. We want to estimate
\begin{equation*}
   \Delta_n  =\PP^x(\tau>n)-\PP^x(\tau=\infty)= \PP^x(n<\tau <\infty).
\end{equation*}
By the geometry of $K$, we have $\{\tau<\infty\}=\cup_{u\in S}\{\sigma_{u}<\infty\}$, where $\sigma_u$ denotes the first exit time from the half-space $D_u$, therefore
\begin{equation*}
   \max_{u\in S}\PP^x(n<\tau, \sigma_{u}<\infty)\leq  \Delta_n \leq \sum_{u\in S}\PP^x(n<\tau, \sigma_{u}<\infty).
\end{equation*}
If $u$ is such that $\PP(X_1 \in D_u)=1$, then the random walk can not leave the half-space $D_u$ and 
$\PP^x(n<\tau, \sigma_{u}<\infty)=0$. So we can rewrite the preceding relation as
\begin{equation*}
\label{eq:encadrementsimple}
\max_{u\in S'}\PP^x(n<\tau, \sigma_{u}<\infty)\leq  \Delta_n \leq \sum_{u\in S'}\PP^x(n<\tau, \sigma_{u}<\infty),
\end{equation*}
where $S'$ is the subset of all $u\in S$ satisfying $\PP(X_1\in D_u)<1$.
We shall see that those simple bounds are sufficient to obtain our estimate on $\Delta_n$. 
Estimates of each term $\PP^x(n<\tau, \sigma_{u}<\infty)$ are obtained in Lemma~\ref{lem:stay-cone-escape-halspace}. Theorem~\ref{thm:inner_drift} then follows immediately. Lemmas~\ref{lem:new_drift_on_gamma} and~\ref{lem:cramer_with_stopping_times} are preparatory material.

\subsection{Turning the drit inside out}
\label{subsec:basicsLaplace}
The Laplace transform of a vector $X=(X^{(1)}, \ldots, X^{(d)})\in\RR^d$ with probability distribution $\mu$ is the function $L$ defined for $t\in\RR^d$ by
\begin{equation*}
     L(t)= \EE \bigl( e^{\sclr{t}{X}}\bigr)=\int_{\RR^d} e^{\sclr{t}{y}} \mu({d}y).
\end{equation*}
It is finite in some neighborhood of the origin if and only if $\EE \bigl( e^{\alpha \Vert X\Vert}\bigr) $ is finite for some $\alpha>0$.
If $L$ is finite in a neighborhood of the origin, say $\overline{B(0,r)}$, then $L$ is infinitely differentiable in $B(0,r)$ and its partial derivatives are given there by
\begin{equation*}
     \frac{\partial L(t)}{\partial t_i}=\EE \bigl( X^{(i)}e^{\sclr{t}{X}}\bigr).
\end{equation*}
Therefore, the expectation $\EE X=(\EE X^{(1)}, \ldots, \EE X^{(d)})$ is the gradient of $L$ at the origin $\nabla L(0)$. Notice that $X$ is centered (i.e.~$\EE X=0$) if and only if $0$ is a critical point of $L$. Since $L$ is a convex function, this means that $0$ is a minimum point of $L$ in $\overline{B(0,r)}$.

Now suppose that $L$ is finite in some ball $\overline{B(t_0,r)}$ and define a new probability measure $\mu_*$ by
\begin{equation}
\label{eq:def/mu*}
   \mu_{*}({d}y)=\frac{e^{\sclr{t_0}{y}}}{L(t_0)}\mu({d}y).
\end{equation}
We will say that $\mu_{*}$ is the $t_0$-changed measure.
The Laplace transform $L_*$ of $\mu_*$ is linked to that of $\mu$ by the relation $L_*(t)=L(t_0+t)/L(t_0)$,
and therefore $L_*$ is finite in some neighborhood of the origin. As a consequence, applying the results above shows that any random vector $X_*$ with distribution $\mu_*$ satisfies:
\begin{itemize} 
\item $\EE \left( e^{\alpha \Vert X_*\Vert}\right)< \infty$ for some $\alpha>0$;
\item $\EE X_*= \nabla L(t_0) / L(t_0)$.
\end{itemize}
For specific points $t_0$ satisfying the equation $L(t_0)=1$, we obtain additional information on the new drift:
\begin{lem}
\label{lem:new_drift_on_gamma}
Assume hypotheses \ref{hyp:A1}, \ref{hyp:A2}, \ref{hyp:A4}, \ref{hyp:A5} above, $K=\cap_{u\in S} D_{u}$, and $m\in K^o$. Then the subset $S'\subset S$ of directions $u$ such that the equation $L(su)=1$ has a solution $s=s_u<0$ is non-empty and equal to the set of all $u\in S$ such that $\PP(X_1\in D_u)<1$.
In addition, for any $u\in S'$, the gradient of $L$ at the point $t_u=s_u u$ satisfies $\sclr{\nabla L(t_u)}{u}<0$.  
\end{lem}
\begin{rem}
In other words, under the $t_u$-changed measure,  the new drift  does not belong to the half-space
$D_u=\{z\in\RR^d \vert \sclr{z}{u}\geq 0\}$.
\end{rem}

\begin{proof}
First, we note that $\sclr{m}{u}>0$ for all non-zero  $u\in K^*$. Indeed, if $C$ is a closed cone, the interior of its dual cone has the following description:
\begin{equation*}
   (C^*)^o=\bigl\{ y \in \mathbb{R}^d \vert \sclr{y}{u}>0 \mbox{ for all }u\in C\setminus\{0\} \bigr\},
\end{equation*}
see Exercise 2.31(d) in \cite{BoVa04} for example. 
Since $K$ is a closed convex cone, it is well known that $(K^*)^*=K$ (see Consequence 1 in \cite{Stu93} or Theorem 14.1 in \cite{Roc70}).
Applying this to $C=K^*$, we see that the interior of $K$ can be expressed as
\begin{equation*}
   K^o=\bigl\{ y \in \RR^d \vert \sclr{y}{u} > 0 \mbox{ for all }u\in K^*\setminus\{0\} \bigr\}.
\end{equation*}
This proves the first assertion. 

Let $u\in S$. By definition one has $K\subset D_u$, thus $u\in K^*$. Consider the partial function of a real variable $\phi(s)=L(su)$. This function is $C^{\infty}$ and strictly convex, since the Laplace transform is $C^{\infty}$ thanks to \ref{hyp:A2} and strictly convex thanks to \ref{hyp:A4}.
Its derivative is given by $\phi'(s)=\sclr{\nabla L(su)}{u}$, hence $\phi'(0)=\sclr{m}{u}>0$.
Based on \cite[Lem.~6]{GaRa16}, the following dichotomy holds:
\begin{itemize}
   \item If $\PP(X_1\in D_u)<1$, then $\lim_{s\to -\infty}\phi(s)=\infty$. Since $\phi$ is strictly convex and satisfies $\phi(0)=1$ and $\phi'(0)>0$, there exists a unique $s<0$ such that 
$\phi(s)=1$. Moreover, at this point, the derivative $\phi'(s)$ must be negative. Hence $L(su)=1$ and $\sclr{\nabla L(su)}{u}<0$. 
   \item If $\PP(X_1\in D_u)=1$, then $\lim_{s\to -\infty}\phi(s)=\PP(\sclr{X_1}{u}=0)<1$. In this case, the equation $\phi(s)=1$ has no solution $s<0$.
\end{itemize}
Since $\PP(X_1\in K)<1$ by \ref{hyp:A5} and $K=\cap_{u\in S}D_u$ with $S$ finite, there is at least one $u\in S$ such that $\PP(X_1\in D_u)<1$. Therefore $S'$ is non-empty.
\end{proof}

\subsection{Change of measure}

Let $t_0$ be given and consider the $t_0$-changed measure \eqref{eq:def/mu*}.
We shall denote by $\PP_*^x$ a probability distribution under which $(S_n)_{n\geq0}$ is a random walk with increment distribution $\mu_{*}$ and started at $S_0=x$. It is easily checked that
\begin{equation}
\label{eq:girsanov}
\EE^x\bigl(f(S_1,S_2,\ldots, S_k)\bigr)  =  L(t_0)^k e^{\sclr{t_0}{x}} \EE_{*}^x\left(f(S_1,S_2,\ldots, S_k)e^{-\sclr{t_0}{S_k}}\right),
\end{equation}
for any non-negative measurable function $f:\RR^k\to[0,\infty)$.

\begin{lem}
\label{lem:cramer_with_stopping_times}
Assume $L(t_0)=1$, and let $\tau$ and $\sigma$ be two stopping times w.r.t.~the natural filtration associated with $(S_n)_{n\geq 0}$. Then
\begin{equation}
\label{eq:cramer_with_stopping_times}
\PP^x(n<\tau\leq\sigma<\infty)=e^{\sclr{t_0}{x}} \EE_{*}^x\left(n<\tau\leq \sigma <\infty, e^{-\sclr{t_0}{S_{\sigma}}}\right).
\end{equation}
\end{lem}
\begin{proof}
For all $k>n$, applying \eqref{eq:girsanov} gives
\begin{align*}
\PP^x(n<\tau\leq\sigma=k) & =e^{\sclr{t_0}{x}} \EE_{*}^x\left(n<\tau\leq \sigma=k, e^{-\sclr{t_0}{S_{k}}}\right)\\
&= e^{\sclr{t_0}{x}} \EE_{*}^x\left(n<\tau\leq \sigma=k, e^{-\sclr{t_0}{S_{\sigma}}}\right).
\end{align*}
By summing over $k>n$, we obtain~\eqref{eq:cramer_with_stopping_times}.
\end{proof}

By specializing this relation to the exit time $\tau$ from the cone $K=\cap_{u\in S}D_u$ and the exit time $\sigma_u$ from one half-space $D_u$, we shall obtain the following: 
\begin{lem}
\label{lem:stay-cone-escape-halspace}
Assume hypotheses \ref{hyp:A1}--\ref{hyp:A5}, $K=\cap_{u\in S}D_u$, and $m\in K^o$. For all $u\in S'$,
\begin{equation*}
   \PP^x(n<\tau,  \sigma_u<\infty)=\rho^n_u B_n,
\end{equation*}
where  $B_n$ satisfies $\sqrt[n]{B_n}\to 1$ and $B_n\to 0$,
and
\begin{equation*}
   \rho_u=\min_{z\in K^*}L(t_u+z)\in(0,1),
\end{equation*}
where $t_u=su$ is the unique solution to $L(su)=1$  with $s<0$.
\end{lem}
\begin{proof}
Fix $u \in S'$ and let $t_u$ be as in the statement of Lemma~\ref{lem:new_drift_on_gamma}.
Then the drift $m_{*}=\nabla L(t_u)$ of the random walk under $\PP_{*}^x$ satisfies $\sclr{u}{m_*}<0$, so that $\sigma_u$ is almost surely finite and the relation \eqref{eq:cramer_with_stopping_times} of Lemma~\ref{lem:cramer_with_stopping_times} gives
\begin{align*}
\PP^x(n<\tau, \sigma_u<\infty) &=e^{\sclr{t_u}{x}} \EE_{*}^x\left(n<\tau, e^{-\sclr{t_u}{S_{\sigma_u}}}\right)\\
    &=e^{\sclr{t_u}{x}} \EE_{*}^x\left(n<\tau, \EE_{*}^{S_n}\left(e^{-\sclr{t_u}{S_{\sigma_u}}}\right) \right)\\
    &=e^{\sclr{t_u}{x}} \EE_{*}^x\bigl(n<\tau, \phi(S_n) \bigr),
\end{align*}
where we have set 
\begin{equation*}
   \phi(y)=\EE_{*}^{y}\left(e^{-\sclr{t_u}{S_{\sigma_u}}}\right)=\EE_{*}^{y}\left(e^{s_u\sclr{-u}{S_{\sigma_u}}}\right).
\end{equation*}
Let us focus on $\phi(y)$ for $y\in K$.
Under $\PP^y_{*}$, the projected random walk $Z_n=\sclr{-u}{S_n}$ is started at $\widetilde{y}=\sclr{-u}{y}\leq 0$ and  has a positive drift $\sclr{-u}{m_*}>0$. The random time $\sigma_u$ corresponds to its first exit time $T$ from the negative half-line $(-\infty, 0]$ and therefore $\phi(y)$ is the expectation of $f(Z_T)$ when the random walk is started at $\widetilde{y}\leq 0$, where $f(t)=e^{s_ut}$ is continuous and non-increasing.
So, it follows from Lemma~\ref{lem:overshoot} below that the function $\phi(y)$ is bounded from above and below on the cone $K$ by two positive constants $0<c<C$.
Therefore
\begin{equation*}
ce^{\sclr{t_u}{x}}\, \PP_{*}^x\left(\tau>n\right) \leq  \PP^x(n<\tau  \leq \sigma_u<\infty) \leq Ce^{\sclr{t_u}{x}}\,\PP_{*}^x\left(\tau>n\right).
\end{equation*}

As a last step, we apply Theorem~\ref{thm:exp_rate_and_zero_term} to estimate $\PP_{*}^x\left(\tau>n\right)$. Let us see why the hypotheses of this theorem are satisfied under $\PP_{*}$, i.e.~by our new random walk with increment distribution $\mu_{*}({d}y)=e^{\sclr{t_u}{y}}\mu({d}y)$ and Laplace transform $L_{*}(t)=L(t_u+t)$:
\begin{itemize}
   \item[\ref{hyp:A1}] The cone hasn't changed.
   \item[\ref{hyp:A2}] The random walk is truly $d$-dimensional, since this condition depends only on the support of  $\mu_{*}$ (it should not be included in any linear hyperplane)\ and its support is exactly the same as that of $\mu$.
   \item[\ref{hyp:A3}] For the same reason, the new random walk inherits from the original random walk the property that it can reach the interior of the cone. This can be seen via~\eqref{eq:girsanov}.
   \item[\ref{hyp:A4}] The Laplace transform  $L_{*}(t)=L(t_u+t)$ is finite everywhere, since $L$ is finite everywhere.
   \item The new drift $m_{*}$ is not in $K$, since $\sclr{m_{*}}{u}<0$ and $K\subset D_u$.
   \item Finally, it remains to check that $L_{*}$ is coercive on the dual cone $K^*$. Fix $v\in K^*$ and recall from \cite[Lem.~6]{GaRa16} that $\lim_{t\to \infty}L_{*}(tv)=\infty$ if and only if the support of $\mu_{*}$ is not included in $-D_v$. But $\mu$ and $\mu_{*}$ have the same support which is not included in $-D_v$, for else the original drift $m$ would also be in $-D_v$. This is impossible since $m\in K^o$ and $-D_v\cap K^o=\emptyset$.
\end{itemize}
It follows from Theorem~\ref{thm:exp_rate_and_zero_term} that
\begin{equation*}
   \PP_{*}^x\left(\tau>n\right)=\rho^n_u B_n,
\end{equation*}
where  $B_n$ satisfies $\sqrt[n]{B_n}\to 1$ and $B_n\to 0$,
and
\begin{equation*}
   \rho_u=\min_{z\in K^*}L(t_u+z)\in(0,1).
\end{equation*}
($\rho_u<1$ since $m_{*}\not\in K$, see Theorem~\ref{thm:exp_rate_and_zero_term}.)
This concludes the proof of the lemma.
\end{proof}

We end this section with a result on overshoots of a random walk that is uniform w.r.t.~the starting point. 

\begin{lem}[Overshoot]
\label{lem:overshoot}
Let $(Z_n)_{n\geq 0}$ be a one-dimensional random walk with integrable i.i.d.\ increments $Y_k$ satisfying $\EE Y_k >0$. Let $T$ denote the first exit time from the half-line $(-\infty, 0]$. For any continuous, non-increasing function $f:[0,\infty)\to (0,\infty)$, we have
\begin{equation*}
   0<\inf_{x\leq 0} \EE^xf(Z_{T})\leq \sup_{x\leq 0} \EE^xf(Z_{T})<\infty.
\end{equation*}
\end{lem}

\begin{proof}
Since the random walk has positive drift, the exit time $T$ is almost surely finite, furthermore $\EE^x T<\infty$ for all $x\leq 0$. Since $Z_T>0$  and $f$ is non-increasing, we have  $\EE^x f(Z_{T})\leq f(0)$ for all $x\leq 0$. This proves the rightmost inequality of Lemma~\ref{lem:overshoot}. 

We now turn to the leftmost inequality side, which is the difficult part. Since
\begin{equation*}
   \EE^x f(Z_{T}) \geq \EE^x \left( f(Z_{T}), Z_T\leq a \right)\geq f(a) \PP^x(Z_T\leq a),
\end{equation*}
it suffices to find $a>0$ such that
\begin{equation*}
   \inf_{x\leq 0}\PP^x(Z_T\leq a)>0.
\end{equation*}
We shall first exhibit a lower bound when $x$ remains in a bounded interval $[b,0]$, $b<0$.
To do this, simply write
\begin{align*}
\PP^x(Z_T >a ) &= \sum_{n\geq 1} \PP^x(Z_n > a, T=n) \\
    & \leq \sum_{n\geq 1} \PP^x( Y_n > a, T>n-1) = \PP(Y_1 > a) \EE^x T.
\end{align*}
Let $b$ be an arbitrary negative number. Since $\EE^x T$ is non-increasing as $x\uparrow 0$, we have
\begin{equation*}
   \sup_{b\leq x\leq 0}\PP^x(Z_T >a) \leq \PP(Y_1 > a) \EE^b T.
\end{equation*}
Choosing $a>0$ such that $\PP(Y_1 >a) \EE^b T<1$, we obtain $\inf_{b\leq x\leq 0}\PP^x(Z_T \leq a)>0$.

It remains to lower bound $\EE^x f(Z_{T})$ for large negative $x$. To do this, we use a well-known consequence of the renewal theorem, which asserts that the overshoot $Z_T$ above $0$ converges in distribution as the initial state
$x$ goes to $-\infty$. To state this result precisely, let us introduce the ladder epochs $(T_k)_{k \geq 0}$ defined by
\begin{equation*}
   T_0=0 \quad\mbox{ and }\quad  T_{k+1}=\inf\{n> T_k \vert Z_n>Z_{T_k}\}, \quad k\geq 0,
\end{equation*}
and the corresponding ladder heights $H_k=Z_{T_k}$. It is clear that our exit time $T$ must occur at one of the ladder epochs, hence 
$Z_T=H_{\tau_+}$, where $\tau_+=\inf\{n>0 \vert H_n>0\} $. Now 
\begin{equation*}
   H_n=Z_0+\sum_{k=1}^n(Z_{T_k}-Z_{T_{k-1}}),
\end{equation*}
where the random variables  $Z_{T_k}-Z_{T_{k-1}}$ are i.i.d.\ and positive. 
According to \cite[XI.4, Eq.~(4.10), p.~370]{Fel71}, the overshoot $H_{\tau_+}$ above $0$ of the renewal process $(H_n)_{n\geq 0}$ converges in distribution as $x\to -\infty$. More precisely,
if the distribution of the ladder heights increments $Z_{T_k}-Z_{T_{k-1}}$ is non-arithmetic\footnote{A probability  distribution is arithmetic if it is concentrated on $\lambda \mathbb{Z}$ for some $\lambda>0$. In this case, the largest $\lambda$ with this property is called the \textit{span}.}, then for any bounded and continuous function $f:[0,\infty)\to \RR$, 
\begin{equation*}
   \lim_{x\to -\infty} \EE^xf(H_{\tau_+})=\frac{1}{\nu}\int_{0}^{\infty}f(t) G(t) dt,
\end{equation*}
where $G$ is the survival function of $Z_{T_k}-Z_{T_{k-1}}$ and $\nu$ is its expectation. Since the variable $Z_{T_k}-Z_{T_{k-1}}$ is positive and our function $f$ is also positive, the integral above is positive.


In case the distribution of $Z_{T_k}-Z_{T_{k-1}}$ is arithmetic with span $\lambda>0$, then
\begin{equation*}
   \lim_{n\to -\infty} \EE^{n\lambda}f(H_{\tau_+})=\frac{\lambda}{\nu}\sum_{k=0}^{\infty}f((k+1)\lambda) G(k\lambda).
\end{equation*}
The proof is exactly as that of \cite[XI.4, Eq.~(4.10), p.~370]{Fel71}, 
when specialized to the case of arithmetic distributions, see \cite[XI.1, Eq.~(1.19), p.~362]{Fel71}\footnote{Note that there is a misprint in \cite[XI.1, Eq.~(1.19), p.~362]{Fel71}: for $x\in [0,\lambda)$, the indices $j$ should start at $0$.}. Here again the limit is positive, since at least $G(0)>0$. So, there exists $n_0<0$ such that $\inf_{n<n_0}\EE^{n\lambda}f(H_{\tau_+})>0$. Now, for $n<0$ and $x\in ((n-1)\lambda, n\lambda]$, the exit time $\tau_+$ when the process is started at $x$ is identical to that when started at $n\lambda$, hence
\begin{equation*}
    \EE^{x}f(H_{\tau_+})=\EE^{n\lambda}f(H_{\tau_+}-(n\lambda-x))\geq \EE^{n\lambda}f(H_{\tau_+}).
\end{equation*}
Therefore $\inf_{x<n_0\lambda}\EE^{x}f(H_{\tau_+})>0$.
\end{proof}

\appendix

\section{Properties of pyramids and proof of Remark~\ref{rem:choice}}
\label{sec:supp_mat}

The main objective of the appendix is to prove that the value of $\rho$ in Theorem~\ref{thm:inner_drift} doesn't depend on the set $S$ as long as $K=\cap_{u\in S} D_{u}$, as mentioned in Remark~\ref{rem:choice}. Along the way of showing this, we shall recall several statements on pyramids. To make the paper self-contained and for the convenience of the reader, we will briefly prove these technical results.

Recall that a subset $K$ of $\mathbb{R}^d$ is a closed pyramid if there is a finite subset $S$ of the sphere $\mathbb{S}^{d-1}$ such that
$K=\cap_{u\in S} D_u$, where
$D_u$ stands for the closed half-space $\{y\in\RR^d \vert \sclr{y}{u}\geq 0\}$, see Figure~\ref{fig:examples_introduction}. In other words, $K$ is a finite intersection of homogeneous closed half-spaces. A closed pyramid is clearly a closed convex cone.

In this section, we will show (Lemma~\ref{lem:pyramids_carac})\ the following straightforward-looking result: 
if $K$ is a closed pyramid with non-empty interior\footnote{The result is not true in general if the interior of $K$ is empty. Consider for example the cone $K=\{0\}\times [0,\infty)$ in $\mathbb{R}^2$. It can be written as $\cap_{u\in S}D_u$, with $S=\{(1,0), (-1,0), (0,1)\}$,
and this $S$ is minimal w.r.t.~cardinality, but $\widetilde{S}=\{(1,0), (-1,0), (1/\sqrt{2},1/\sqrt{2})\}$ is adapted to $K$ as well.}, there is a finite set $S_0$ such that $K=\cap_{u\in S_0}D_u$ and $S_0$ is minimal w.r.t.~inclusion, i.e.
\begin{equation*}
    K=\cap_{u\in S}D_u\Rightarrow S_0\subset S.
\end{equation*}
From this, we will deduce (Lemma~\ref{lem:wheretofindthemax}) that 
\begin{equation}
\label{eq:explainingthemax}
  \max_{u\in S'}\min_{z\in K^*}L(t_u+z)=\max_{u\in S_0'}\min_{z\in K^*}L(t_u+z),
\end{equation}
thus explaining why the expression for $\rho$ in Theorem~\ref{thm:inner_drift} doesn't depend on the particular choice of the set $S$, provided that $K=\cap_{u\in S} D_{u}$.

\subsection{Some facts about pyramids}

To begin, let $K$ be any non-empty closed convex cone. It is well known that $K$ is the intersection of the homogeneous closed half-spaces $D_u$ which contain it (see \cite[Cor.~11.7.1]{Roc70}). Since the condition $K\subset D_u$ is clearly equivalent to $u\in K^*$, it follows that
\begin{equation*}
   K=\bigcap_{u\in K^*} D_u.
\end{equation*}

From now on, we will assume $K$ has a \textit{non-empty interior} and focus on the structure of the dual cone $K^*$.
Since $K$ contains a ball $B(x_0, r)$, any non-zero $y\in K^*$ satisfies $\sclr{y}{x_0+ru}\geq 0$ for all $u\in\mathbb{S}^{d-1}$, and this condition is equivalent to $\sclr{y}{x_0}\geq r \Vert y\Vert$. This implies that the cone $C=K^*$ is \textit{salient}, i.e.~there exists $v\in \mathbb{S}^{d-1}$ such that $\sclr{y}{v}>0$ for all $y\in C\setminus\{0\}$.

Salient convex cones have the following property:

\begin{lem}
\label{lem:cone_generated_by_its_section}
If a closed convex cone $C$ is salient, then there exists an affine hyperplane $H$, not containing the origin, such that $C\cap H$ is compact and generates $C$, i.e.~$C=\{\lambda y \vert \lambda \geq 0, y\in C\cap H\}$).
\end{lem}
\begin{proof}
The set $C\cap \mathbb{S}^{d-1}$ is compact and generates $C$. Consider the affine hyperplane $H=v+[v]^{\perp}$, where $v$ is as in the definition of a salient cone. The mapping
$\phi:C\cap \mathbb{S}^{d-1}\to C\cap H$ given by $\phi(y)=y/\sclr{y}{v}$ is continuous and onto, therefore $C\cap\mathbb{S}^{d-1}$ is compact. Moreover $C\cap H$ generates $C$, since any non-zero $y$ in $C$ can be written as $y =\lambda \phi(y)$ with $\lambda=\sclr{y}{v}$.
\end{proof}

\begin{figure}
    \centering
    \includegraphics[scale=0.3]{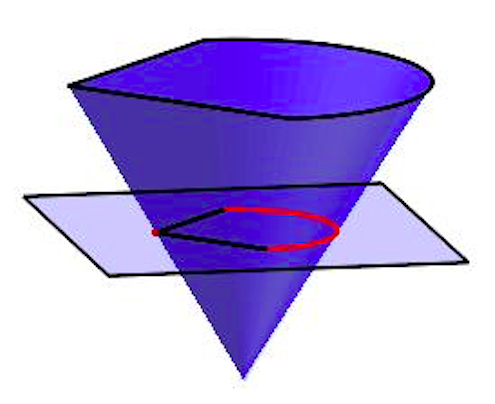}
    \vspace{-5mm}
    \caption{A salient cone $C$ and the extreme points (in red) of $C\cap H$}
    \label{fig:salient_cone}
\end{figure}

Take $H$ as in the lemma above and denote by $\E$ the set of extreme points of the convex set $C\cap H$ (see Figure~\ref{fig:salient_cone}). Since $C\cap H$ is compact, it follows from Minkowski-Steinitz Theorem (see \cite[Cor.~18.5.1]{Roc70})\ that any of its points $a$ can be written as a convex combination of extreme points:
\begin{equation*}
   a=\sum_{i=1}^p \lambda_i a_i,
\end{equation*}
with $a_i\in \E$, $\lambda_i\geq 0$ and $\sum_{i=1}^p \lambda_i=1$. 
From this we define the set of extremal directions of $C$ as 
\begin{equation*}
   \E_d(C)=\big\{a/ \Vert a\Vert \,\vert\, a \in \E \bigr\}.
\end{equation*}
It can be seen that this set does not depend on the particular choice of $H$, and has the following intrisic characterisation: $u\in \mathbb{S}^{d-1}$ is an extremal direction of $C$ if and only if it can not be written as a combination $\alpha a+\beta b$ with $\alpha,\beta >0$ and $a,b\in C$, $a\not=b$.
Since $C\cap H$ generates $C$, it follows that any point $u$ in $C$ may be expressed as
\begin{equation*}
   u=\sum_{i=1}^p \mu_i u_i,
\end{equation*}
with $u_i\in \E_d(C)$ and $\mu_i\geq 0$. This implies $D_u\supset D_{u_1}\cap D_{u_2} \cap \cdots \cap D_{u_p}$.

Applying this to the salient cone $K^*$, we obtain the following representation for $K$:
\begin{equation}
\label{eq:minimal_rep_pyramid}
K=\bigcap_{u\in K^*} D_u=\bigcap_{u\in \E_d(K^*)} D_u.
\end{equation}

\begin{lem}
\label{lem:pyramids_carac}
A closed convex cone $K$ with non-empty interior is a polyhedral cone (or pyramid) if and only if the set $\E_d(K^*)$ is finite.
In this case, for any set $S$  such that $K=\cap_{u\in S}D_u$, we have
\begin{enumerate}
\item $\E_d(K^*)\subset S$,
\item any vector in $S$ is a non-negative linear combination of vectors of $\E_d(K^*)$.
\end{enumerate}
\end{lem}
\begin{proof}
If $\E_d(K^*)$ is finite, then the representation \eqref{eq:minimal_rep_pyramid} shows that $K$ is a finite intersection of half-spaces. 
Conversely, suppose $K=\cap_{u\in S} D_u$ with $S$ finite and let $T\subset S$ be such that $K=\cap_{u\in T} D_u$ and minimal with respect to cardinality.  Write $T=\{u_1,u_2,\ldots,u_p\}$ and consider the set 
\begin{equation*}
   A=\left\{\sum_{i=1}^p \lambda_i u_i \vert  \lambda_1,\lambda_2, \ldots,\lambda_p\geq 0\right\}.
\end{equation*}
This set is clearly a convex cone, and it can be seen that it is a closed set (see \cite[Prop.~1]{Stu93} or  \cite[Thm~14.1]{Roc70}).
It is straightforward that $y$ belongs to $A^*$ if and only if $\sclr{y}{u_i}\geq 0$ for all $1\leq i\leq  p$. In other words,
\begin{equation*}
    A^*=\bigcap_{i=1}^p D_{u_i}=K.
\end{equation*}
Since $A$ is a closed convex cone, $A=(A ^*)^*=K^*$, see Consequence~1 in \cite{Stu93} or \cite[Thm~14.1]{Roc70}. Now, it is easily seen, using the intrisic characterisation of an extremal direction together with the minimality of the set $T$, that any extremal direction of $A$ must be one of the $u_i$'s. Therefore $\E_d(K^*)$ is a subset of $T$, hence finite. (In fact, by minimality $T=\E_d(K^*)$.)

Any vector in $S$ is a non-negative linear combination of vectors of $\E_d(K^*)$, as $S$ is necessarily a subset of $K^*$.
\end{proof}

\subsection{Proof of Remark~\ref{rem:choice}}

We first prove a technical result, which exactly captures the situation encountered in Theorem~\ref{thm:inner_drift}.

\begin{lem}
\label{lem:wheretofindthemax}
Let $K$ be a closed pyramid with non-empty interior and $F, G: K^* \to \RR$ two convex functions such that:
\begin{itemize}
\item $F(0)=1$ and for all $u\in K^* \setminus \{0\}$, either the equation $F(su)=1$ with $s>0$ has exactly one solution or $F(su)<1$ for all $s>0$;
\item $G$ is non-increasing along the rays of the cone $K^*$.
\end{itemize}
Let $M=K^* \cap \{ F=1\}\setminus\{0\}$. If $M$ is non-empty, the maximum of $G$ on $M$ is reached at some point $z$ in 
$M$ such that $u=z/\Vert z\Vert$ is an extremal direction of $K^*$.
\end{lem}
\begin{proof}
Let $A$ be the collection of $u\in \E_d(K^*)$ such that the equation $L(su)=1$ with $s>0$ has exactly one solution, and let $B=\E_d(K^*)\setminus A$. By hypothesis, if $u$ belongs to $B$, then $L(su)<1$ for all $s>0$. We first show that $A$ is non-empty, so that there exists at least one point $z=su$ in $M$ such that $z/\Vert z\Vert \in \E_d(K^*)$. To see this, let $x$ be any point in $M$. As an element of $K^*$, it can be written as a linear combination
\begin{equation}
\label{eq:decomposition_du_point}
x=\sum_{u\in \E_d(K^*) } \lambda_u u
\end{equation}
with  $\lambda_u\geq 0$.
Assume $A$ is empty and let $n$ denote the cardinality of $B$. Then, by convexity of $F$, 
\begin{equation*}
   F(x)=F\left(\frac{1}{n}\sum_{u\in B} n\lambda_u u\right)\leq \frac{1}{n}\sum_{u\in B} F(n \lambda_u u)<1.
\end{equation*}
This contradicts with $F(x)=1$, hence $A$ is non-empty.

For each $u\in A$, denote by $x_u$ the unique point such that $F(x_u)=1$ and $x_u=su$ with $s>0$. Starting from~\eqref{eq:decomposition_du_point}, we can decompose $x$ as follows:
\begin{equation*}
   x=\sum_{u\in A} \mu_u x_u + \sum_{u\in B} \lambda_u u,
\end{equation*}
where $\mu_u\geq 0$ and $\lambda_u\geq 0$. Set $\mu=\sum_{u\in A}\mu_u$, $\lambda=\sum_{u\in B} \lambda_u$ and $c_\epsilon=\mu+\epsilon \lambda$, for a parameter $\epsilon>0$. Repeating the argument above shows that at least one of the $\mu_u$ is positive, hence $\mu>0$ and $c_\epsilon>0$.
By convexity of $F$, we have
\begin{align*}
F(x/c_\epsilon) &=F\left(\sum_{u\in A} \frac{\mu_u}{c_\epsilon} x_u + \sum_{u\in B} \frac{\epsilon\lambda_u}{c_\epsilon} (u/\epsilon)\right)\\
&\leq \sum_{u\in A} \frac{\mu_u}{c_\epsilon} F(x_u) + \sum_{u\in B} \frac{\epsilon\lambda_u}{c_\epsilon} F(u/\epsilon)\leq 1.
\end{align*}
 But, by hypothesis, the function $\phi(s)= F(sx)$ is convex and the equation $\phi(s)=1$ has only two non-negative solutions, namely $s=0$ and $s=1$, thus $\phi(s)>1$ for all $s>1$. Therefore $c_\epsilon$ must be larger than one. 
So, since $G$ is non-increasing along the rays of $K^*$, the value $G(x)$ is less than $G(x/c_\epsilon)$, and using the convexity exactly as above we obtain
\begin{align*}
G(x) & \leq \sum_{u\in A} \frac{\mu_u}{c_\epsilon} G(x_u) + \sum_{u\in B} \frac{\epsilon\lambda_u}{c_\epsilon} G(u/\epsilon)\\
& \leq \frac{\mu}{c_\epsilon} \max_{u\in A} G(x_u) + \frac{\epsilon\lambda}{c_\epsilon} \max_{u\in B}G(u/\epsilon).
\end{align*}
As $\epsilon$ goes to $0$, $c_{\epsilon}$ goes to $\mu$ and $\max_{u\in B}G(u/\epsilon)$ remains bounded (by $\max_{u\in B}G(u)$ for $\epsilon<1$ for example). Therefore, letting $\epsilon\to 0$ in the inequality above leads to
\begin{equation*}
   G(x)\leq \max_{u\in A} G(x_u).
\end{equation*}
This proves the lemma.
\end{proof}

As an application, let $K=\cap_{u\in S}D_u$ be a closed pyramid with non-empty interior and $L$ be the Laplace transform of the distribution $\mu$ satisfying the hypotheses of Theorem~\ref{thm:inner_drift}.
Set 
\begin{equation*}
    F(x)=L(-x) \quad \text{and}\quad  G(x)=\inf_{z\in K^*}L(-x+z).
\end{equation*}
The proof of Lemma~\ref{lem:new_drift_on_gamma} shows that the function $F$ satisfies the first hypothesis of Lemma~\ref{lem:wheretofindthemax} and that $M$ is non-empty. Now $G$ is also a convex function (because $L$ is convex). To see why it is non-increasing along the rays of $K^*$, fix $v\in K^*$ and $a<b$. Since $K^*$ is a convex cone, it is also a semi-group, therefore $(b-a)v+K^*\subset K^*$, and adding $-bv$ on each side leads to $-av+K^*\subset -bv+K^*$. It follows that $G(av)\geq G(bv)$.

Now define $T$ as the set of all $u$ in $K^*$ such that the equation $F(su)=1$ with $s>0$ has exactly one solution, which we write $-\alpha_u$ with $\alpha_u<0$. Then $t_u=\alpha_u u$ denotes the same point as in Theorem~\ref{thm:inner_drift} because $L(\alpha_u u)=F(-\alpha_u u)=1$.  The set $M$ in Lemma~\ref{lem:wheretofindthemax} is equal to $\{-t_u | u\in T\}$ and the lemma asserts that 
\begin{equation*}
   \max\{G(-t_u) | u\in T\}= \max\{G(-t_u) | u\in T\cap \E_d(K^*)\}.
\end{equation*}
More explicitly,
\begin{equation*}
   \max_{u\in T} \inf_{z\in K^*} L(t_u+ z)= \max_{u\in T\cap \E_d(K^*)} \inf_{z\in K^*} L(t_u+ z).
\end{equation*}
But the set $S'$ in Theorem~\ref{thm:inner_drift} satisfies
\begin{equation*}
    T\cap \E_d(K^*)\subset S'\subset T,
\end{equation*}
hence
\begin{equation*}
   \max_{u\in S'} \inf_{z\in K^*} L(t_u+ z)= \max_{u\in T\cap \E_d(K^*)} \inf_{z\in K^*} L(t_u+ z).
\end{equation*}
This explains \eqref{eq:explainingthemax}.

\section*{Acknowledgments}

We would like to thank Marc Peign\'e for interesting discussions on overshoots of random walks. We warmly thank an anonymous referee for their thorough reading and their interesting comments and suggestions.

 \bibliographystyle{unsrt}

\begin{thebibliography}{99}



\bibitem{BoMi10}
M. Bousquet-M\'elou and M. Mishna (2010).
\newblock Walks with small steps in the quarter plane. 
\newblock{\it Contemp. Math.} {\bf520} 1--39

\bibitem{BoVa04}
S. Boyd and L. Vandenberghe (2004).
\newblock{\it Convex Optimization.}
\newblock Cambridge University Press





\bibitem{Fel71}
{W. Feller} (1971).
{\it An Introduction to Probability Theory and Its Applications, Volume 2. Second edition.}
Wyley, New York


\bibitem{GaRa16}
R. Garbit and K. Raschel (2016).
\newblock On the exit time from a cone for random walks with drift.
\newblock{\it Rev. Mat. Iberoam.} \textbf{32} 511--532

\bibitem{GaRa22}
R. Garbit and K. Raschel (2022).
\newblock The generating function of the survival probabilities in a cone is not rational.
\newblock{\it S\'em. Lothar. Combin.} \textbf{87B} 


\bibitem{MeMi-14}
S. Melczer and M. Mishna (2014).
Singularity analysis via the iterated kernel method.
\textit{Combin. Probab. Comput.} \textbf{23} 861--888

\bibitem{MiRe-09}
M. Mishna and A. Rechnitzer (2009).
Two non-holonomic lattice walks in the quarter plane.
\textit{Theoret. Comput. Sci.} \textbf{410} 3616--3630

\bibitem{Stu93}
M. Studen\'y (1993).
\newblock Convex cones in finite-dimensional real vector spaces.
\newblock{\it Kybernetika} \textbf{29}, no. 2, 180--200

\bibitem{Roc70}
{R. T. Rockafellar} (1970).
{\it Convex Analysis.}
Princeton University Press

\end{thebibliography}
 \makeatletter
 \def\@openbib@code{\itemsep=-3pt}
 \makeatother

\end{document}